\documentclass[12pt]{article}

\usepackage[leqno]{amsmath}
\usepackage{amsfonts}
\usepackage{latexsym}
\usepackage{amssymb}
\usepackage{amscd}
\usepackage{amsthm}
\usepackage{mathtools}
\usepackage{relsize}
\usepackage{xcolor}

\newdimen\headwidth
\newdimen\headrulewidth\label{means}
\headwidth \textwidth

\theoremstyle{plain}
\newtheorem{proposition}{Proposition}
\newtheorem{lemma}[proposition]{Lemma}
\newtheorem{theorem}[proposition]{Theorem}
\newtheorem{corollary}[proposition]{Corollary}
\theoremstyle{definition}

\theoremstyle{remark}

\numberwithin{proposition}{section}
\numberwithin{equation}{section}

\newcommand{\bM}{\mathbb {M}}

\newcommand{\bR}{\mathbb {R}}

\newcommand{\sD}{\mathcal{D}}
\newcommand{\sF}{\mathcal{F}}

\newcommand{\sM}{\mathcal{M}}

\newcommand{\sK}{\mathcal{K}}

\newcommand{\ot}{\overline {t}}
\newcommand{\uu}{\underline {u}}

\title{Regularity of Minimizers for a General Class of Constrained Energies in Two-Dimensional Domains with Applications to Liquid Crystals}

\author{Patricia Bauman, \ Daniel Phillips\\
Department of Mathematics\\
Purdue University\\
West Lafayette, IN 47907, USA}

\date{}

\begin{document}
\maketitle
\begin{abstract}
We investigate minimizers defined on a bounded domain $\Omega$ in $\bR^2$ for singular constrained energy functionals that include Ball and Majumdar's modification of the Landau-de Gennes Q-tensor model for nematic liquid crystals. We prove regularity of minimizers with finite energy and show that their range on compact subdomains of $\Omega$ does not intersect
the boundary of the constraining set. We apply this result to prove that minimizers of the constrained Landau-de Gennes Q-tensor energy for liquid crystals composed of a singular Maier-Saupe bulk term and all elasticity terms with coefficients $L_1, \cdots , L_5$, are $C^2$ in $\Omega$; and their eigenvalues on compact subsets of $\Omega$ are contained in closed subintervals of
the physical range $(-\frac {1}{3},\frac {2}{3})$.
\end{abstract}

\section{Introduction.}\label{s1}
In this paper we consider minimizers to a singular constrained energy functional of the form
\begin{equation}\label{Energy}
J[v]=\int\limits_\Omega(F(v,Dv)+f(v))dx
\end{equation}
where $\Omega$ is a bounded $C^2$ domain in $\bR^n$ and $n\geq 2$.  We assume that $f$ is defined and real-valued on an open, bounded, convex set $\sK$ in $\bR^q$ with $q \geq 1$ and that $f(v) \to \infty$ as $v \to \partial \sK$ for $v$ in $\sK$. We extend the definition of $f$ to all of $\bR^q$ by setting $f(v)= \infty$ for $v$ in $\bR^q\setminus \sK$.
Thus we assume throughout this paper that
{\begin{equation}\begin{cases}\label{Po-St}
f:\sK\to\bR,\quad f\in C^2(\sK),\quad D^2f\ge -MI_q \text{ on  } \sK,\\
\underset{v\to\partial\sK  \atop v\in\sK}\lim f(v)=\infty,
\text{   and   } f(v) = \infty \text{ on } \bR^q\setminus \sK
\end{cases}
\end{equation}
where $M \geq 0.$  We also assume the following structure conditions on $F$:
\begin{equation}\begin{cases} \label{El-St}
F(v,P)=A^{\alpha\beta}_{ij}(v)p^i_\alpha p^j_\beta +B^\alpha_i(v)p^i_\alpha\quad \mathrm{for}\,\, v\in\overline{\sK}, P\in\bM^{q\times n},\\
\\
 A^{\alpha\beta}_{ij}(v)p^i_\alpha p^j_\beta \ge \lambda |P|^2 \,\,\, \text{  for}\,\, v\in \overline{\sK}, P \in\bM^{q\times n},\\
\\
\text{where}\,\, A^{\alpha\beta}_{ij}, B^\alpha_i\in  C^{2}(\overline{\sK})\,\,\text{ and}\,\, \lambda>0.
\end{cases}\end{equation}
 We use the convention in this paper  that repeated indicies are summed. In this case $i$  and  $j$ go from  $1$  to  $n$  and
 $\alpha$ and $\beta$ go from  $1$  to  $q$. Here $\bM^{q\times n}$ denotes the set of $q\times n$ real-valued matrices.}  We define
$$M_1=max[\{sup_{v \in \overline \sK}|A_{ij}^{\alpha \beta}(v)|\},\{sup_{v \in \overline \sK}|B_{i}^{\alpha}(v)|\}]  $$
and $M_2\equiv max \{\|A_{ij}^{\alpha \beta}\|_{C^2(\overline \sK)},
\|B_{i}^{\alpha}\|_{C^2(\overline \sK)}\}.$

The energy functional $J$ is defined for all $v$ in
\[
H^1(\Omega;\overline{\sK})=\{v\in H^1(\Omega;\bR^q): v(x)\in \overline{\sK}\,\, \mathrm{ almost\  everywhere\  in }\,\, \Omega\}.\]
By assumption \eqref{Po-St} on $f$, if $u \in H^1(\Omega;\overline{\sK})$ and $J[u] < \infty$, then $u(x)\in\sK$ for almost every $x$ in $\Omega$.

Given $\underbar{u}$ in $H^1(\Omega;\overline{\sK})$ such that $J[\underbar{u}] < \infty$, it follows from direct methods in the calculus of variations (see \cite{G}) that minimizers exist in the space $A_{\uu}=\{v\in H^1(\Omega;\overline{\sK}): v-\uu\in H^1_0(\Omega;\bR^q)\}$.  In this paper we will refer to such minimizers as {\it finite energy minimizers} of $J$ in $\Omega$.

A question of interest for applications is whether minimizers $u$ of $J[\cdot]$ are smooth and whether they satisfy $u(x) \in \sK$ for all $x$ in $\Omega$.  One of the  difficulties in analyzing their regularity
is to find finite energy variations in $H^1(\Omega;\overline{\sK})$ from which one can extract useful information about their properties.

Our results are for $n=2$. We prove the following main theorem:

\begin{theorem}\label{uthm}
Assume $\Omega$ is a bounded $C^2$ domain in $\mathbb{R}^n, n=2, {\sK}$ is an open bounded, convex set in $\bR^q$, $q\geq1$, and \eqref{Po-St} and \eqref{El-St} hold. If $u\in H^1(\Omega;\overline{\sK})$ is a finite energy minimizer for $J$ in $\Omega$, then $u$ is in $C^{2,\delta}(\Omega)$ for all $0<\delta<1$, $u(\Omega) \subset \sK$, and $u$ satisfies the equilibrium equation
\begin{equation}\label{Eq}
\rm{div}\, F_P(u, Du) - F_u(u,Du)=f_u(u)\,\, \text{ on } \Omega.
\end{equation}
Moreover, if $\Omega'$ is an open set in $\Omega$ such that $\Omega' \subset\subset\Omega$, then $\text{dist }(u(\Omega'),\partial \sK) \geq c > 0$ where $c$ depends only on $J[u]$, $\Omega$, $\text{dist }(\Omega',\partial \Omega)$, $M_2$, $M$, and $\lambda$.
\end{theorem}

Minimizers $u$ of (\ref{Energy}) were investigated by Evans, Kneuss and Tran in \cite{E}.  Assuming that $n \geq 2$, $F=F(v,Dv)$ satisfies certain growth conditions and is uniformly strictly quasi-convex,
and $f$ satisfies (\ref{Po-St}) with $M=0$, they proved the following partial regularity result: there is an open subset $\Omega_0$ of $\Omega$ such that $|\Omega\setminus\Omega_0|=0$, $u \in C^2(\Omega_0)$, and $u(\Omega_0)\subset\sK$. In particular it follows that $u$ satisfies the equilibrium equation (\ref{Eq}) on $\Omega_0$ in this case.

They also proved in \cite{E} that if, in addition, $F=F(Dv)$, $F(\cdot)$ is convex, and $f(v)$ is smooth and convex on $\sK$,  then a finite energy minimizer $u$ is in $H^2_{loc}(\Omega)$. We use a similar approach here for part of our analysis. They considered variations $u+t \phi_k$ where $\phi_k$ is given by
\begin{align*}
\phi_k(x; h)&=[\zeta^2(x)u(x+h e_k)+\zeta^2(x-h e_k) u(x-h e_k)-(\zeta^2(x)+\zeta^2(x-h e_k))u(x)]h^{-2}\\
&=\nabla^{-h}_k(\zeta^2\nabla_k^h u)\quad \rm{for}\,\, 1\le k\le n,
\end{align*}
for $\zeta\in C^2_c(\Omega)$ and $\nabla^h_k w(x)=h^{-1}[w(x+he_k)-w(x)]$ for $h\ne 0$ and sufficiently small.
They showed that for $0\le t<\ot(h)$, $u(x)+t \phi_k(x)\in\overline\sK$.
Using the definition of $\phi_k$ and the convexity of $f$ they proved that
\begin{equation*}\label{Est-0}
\int\limits_\Omega f(u+t\phi_k)dx \le \int\limits_\Omega f(u)dx\quad\rm{for}\,\, 0\le t<\ot(h).
\end{equation*}
and hence
$$0\le J(u+t \phi_k)-J(u)\le \int\limits_\Omega(F(D(u+t \phi_k))-F(Du)) dx.$$
Dividing by $t$ and letting $t\to 0$ gives
$$0\le\int\limits_\Omega(F_P( Du): D\phi_k) dx \text{ for } 1\le k\le n.$$
For $F=F(Dv)$ this inequality leads to $u\in H^2_{\rm{loc}}(\Omega)$.  (See \cite{E}, Thm 4.1.)

Here we have $F=F(v,Dv)$ and $f$ is not convex. However, our assumption (\ref{Po-St}) implies that
$$f_0(v)\equiv f(v) + \frac{M}{2}|v|^2$$
is a convex function on $\sK$ and by \cite{E}

\begin{equation}\label{Est-0-2}
\int\limits_\Omega f_0(u+t\phi_k)dx \le \int\limits_\Omega f_0(u)dx\quad\rm{for}\,\, 0\le t<\ot(h).
\end{equation}
Using this we can argue just as above to show that a minimizer $u$ of J satisfies
\begin{equation}\label{Est-1-1}
0\le\int\limits_\Omega(F_P(u, Du): D\phi_k +F_u(u,Du)\cdot \phi_k -M u\cdot \phi_k) dx.
\end{equation}
From this and \eqref{El-St} it follows that
\begin{equation}\label{Est-7}
\frac{\lambda}{2}\int\limits_{\Omega} |\nabla^h Du|^2\zeta^2\ dx \le C_0\int\limits_{\Omega}|\nabla^h u|^2 |Du|^2\zeta^2\ dx + C_1
\end{equation}
where $\lambda$ is the constant defined in our assumption (\ref{El-St}).

To obtain an $H^2_{{loc}}$ estimate, we additionally need that $u$ is continuous and a second inequality.  In Section \ref{s2}, we show that when $n=2$, our assumptions \eqref{Po-St}, \eqref{El-St}, and a result in \cite{BP} imply that finite energy minimizers of $J$ are continuous in $\Omega$.  We then construct additional variations $u+tw_l$ in $H^1(\Omega;\overline{\sK})$ so that  $w_l$ satisfies
\eqref{Est-0-2} and \eqref{Est-1-1},
with $\phi_k$ replaced by $w_l$.  We use this to prove a second inequality, from which we obtain the $H^2_{{loc}}$ regularity of minimizers.  In Section \ref{s3} we use this result to prove Theorem \ref{uthm}.

Our formulation of the constrained energy (\ref{Energy}) and the assumptions \eqref{Po-St} and \eqref{El-St} is motivated by the constrained Landau-de Gennes Q-tensor energy for nematic liquid crystals.  This energy is given by
\begin{equation}\label{Energy2}
I_{LdG}[Q]=\int\limits_\Omega\bigl [G(Q,DQ) + \Psi_b(Q)] dx.
\end{equation}
where
\begin{align*}
G(Q,DQ)  =  &  L_1 |\nabla Q|^2  +  L_2 \cdot D_{x_j} Q_{ij} \cdot D_{x_k}Q_{ik} + L_3 \cdot D_{x_j} Q_{ik}\cdot D_{x_k}Q_{ij}\\
& + L_4 \cdot Q_{\ell k}\cdot D_{x_\ell} Q_{ij}\cdot D_{x_k}Q_{ij} + L_5 \cdot\epsilon_{\ell j k}\cdot
Q_{\ell i}\cdot D_{x_j} Q_{ki}\\
& \equiv L_1 I_1 +  L_2 I_2 +  L_3 I_3 +  L_4 I_4 +  L_5 I_5
\end{align*}
and $\Psi_b(Q) = T f_{ms}(Q) - \kappa |Q|^2$.
The constants, $L_1,L_2,L_3,L_4,$ and $L_5$ are  material-dependent elastic constants, $T$ and $\kappa$  are positive constants, and $\epsilon_{\ell j k}$ is the Levi-Civita tensor.  The function
$f_{ms}(Q)$ is a specific function (called the Maier-Saupe potential) defined on
\begin{align*}
\sM =  \{Q \in \mathbb{M}^{3\times 3}:  &  Q=Q^t, \text{ tr } Q=0, \text{ and } -\frac{1}{3} < \lambda(Q) < \frac{2}{3}\\
&\text{    for all eigenvalues } \lambda(Q) \text{ of } Q\}.
\end{align*}
It is defined abstractly using probability densities on a sphere that represent possible orientations of liquid crystal molecules.  (See Section \ref{s4} for the definition of $f_{ms}$.)   The bulk term $\Psi_b$ and the Maier-Saupe potential $f_{ms}$ were introduced and investigated in the papers \cite{BM} by Ball and Majumdar and \cite{K} by Katriel, Kventsel, Luckhurst and Sluckin.
It is known that $f_{ms}$ is convex. Moreover, $f_{ms}$ is bounded below and $f_{ms}(Q) \to \infty$ for $Q$ in $\sM$ with $Q \to \partial \sM$; hence the same is true for $\Psi_b(Q)$.  As in (\ref{Po-St}), we set $\Psi_b(Q)=\infty$ for $Q \in S_0\setminus \sM$ where $S_0=\{Q\in \bM^{3\times 3}: Q=Q^t, tr\, Q=0\}.$  Thus $\Psi_b$ blows up at $\partial \sM$ as in $(1.2.2)$,
with $f$ replaced by $\Psi_b$ and $\sK$ replaced by $\sM$.  It follows that finite energy minimizers $Q$ of $I_{LdG}$ satisfy $Q(x) \in \sM$ almost everywhere in $\Omega$, so that the eigenvalues of $Q(x)$ are in $(-\frac{1}{3},\frac{2}{3})$ almost everywhere in $\Omega$.
Conditions on $L_1,\cdots, L_4$ have been identified so that minimizers of $I_{LdG}$ exist in
\[
A_{Q_0} \equiv \{Q \in H^1(\Omega;\overline{\sM})=\{Q-Q_0\in H^1_0(\Omega;\bM^{3\times 3})\}\]
provided that $Q_0\in H^1(\Omega;\overline{\sM})$ and that $I_{LdG}[Q_0]<\infty$. (See (\ref{47}) and \cite{L}.)
It was stated in \cite{BM} that for $\Omega$ in $\bR^3$, if
${Q_0}(\overline \Omega)\subset\subset\sM$, minimizers in $A_{Q_0}$ of the energy
$$\int\limits_\Omega\bigl [L_1 |DQ)|^2 + \Psi_b(Q)] dx $$
with $L_1>0$ are smooth in $\Omega$ and valued in $\sM$; thus their eigenvalues are in  $(-\frac{1}{3},\frac{2}{3})$ at all points in $\Omega$. A sketch of a proof of this statement is included in \cite{B1}. (See also \cite{BP}.)
Such minimizers are called "physically realistic." Additional features for minimizers, $\tilde Q$ of $I_{LdG}$ with $\Omega$ in $\bR^3$ were obtained by Geng and Tong in \cite{GT}. In particular, assuming specific conditions on $G(Q, DQ)$ they proved higher integrability properties for $|D\tilde{Q}|.$

The elastic term with coefficient $L_4$ in $I_{LdG}$ is called the "cubic term."  When $L_4 \neq 0$, the energy density is quasilinear.  This makes it difficult to analyze the behavior of minimizers in this case.

Physicists have computed the elastic coefficients $L_1,\cdots, L_4$ in terms of the elastic coefficients $K_1, \cdots, K_4$ that account for the elastic energy of splay, twist and bend that occur in the well-known Frank energy density, which models liquid crystals in terms of functions $n=n(x)$ valued in $\mathbb{S}^2$.  They found that $L_4=0$ if and only if $K_1=K_3$, which is nonphysical for many applications.  Thus it is desirable to consider the energy $I_{LdG}$
with $L_4 \neq 0$. It is interesting to note that when $L_4\ne 0$, the unconstrained Landau-de Gennes energy given by \eqref{Energy2} with $\Psi_b(Q)$ replaced by a polynomial  is unbounded from below.  Thus  minimizers do not  exist in general for boundary value problems with this energy. See \cite{BM}.

For $\Omega$ in $\bR^2$, we proved H\"{o}lder continuity of finite energy minimizers in \cite{BP} under general conditions for energy functionals of the form
$$\sF(Q)=\int_\Omega [F_e(Q(x),\nabla Q(x)) + f_b(Q(x))] dx$$
by using harmonic and elliptic replacements to construct finite energy comparison functions.
In particular we established that finite energy minimizers to the quasilinear constrained energy $I_{\text {LdG}}$ in \eqref{Energy2} under the coercivity condition \eqref{47}
are H\"{o}lder continuous in $\Omega$. We also
proved under the additional assumption $L_2=L_3=0$ that finite energy  minimizers for \eqref{Energy2} satisfy the "physicality condition", $Q(x) \in \sM$ for all $x \in \Omega.$  Here we establish this property without requiring  the additional assumption.

In Section \ref{s4} of this paper we describe a connection between the constrained energies $J[u]$ and $I_{LdG}[Q]$.  Using Theorem \ref{uthm}, we prove in Theorem \ref{thm41} that under appropriate coercivity conditions on $L_1,\cdots,L_4$, finite energy minimizers of $I_{LdG}$ with all elasticity terms are in $C^2(\Omega)$; moreover, they satisfy a strong physicality condition:  if $\Omega'$ is an open set such that $\Omega' \subset \subset \Omega$, then $Q(\Omega') \subset \sM$ and $\text{dist} (Q(\Omega'), \partial \Omega) \geq c > 0.$  Thus in compact subsets of $\Omega$, the eigenvalues of minimizers are contained in closed subintervals of the physical range $(-\frac{1}{3},\frac{2}{3})$.

\section{Continuity and $H^2_{\text {loc}}$ estimates for minimizers in two-dimensional domains.}\label{s2}

Assume that $\Omega$ is a bounded $C^2$ domain in $\bR^2$. Let $\Lambda=\{x \in \Omega: u(x) \in \partial \sK\}.$   In this section we will show that finite energy minimizers u of J are locally H\"{o}lder continuous in $\Omega$, $C^2$ in $\Omega \setminus \Lambda$,  and
globally H\"{o}lder continuous in $\overline \Omega$ if their boundary values are sufficiently smooth.  In addition, they are in $H^2_{\text {loc}}(\Omega)$.

Our proof of the first statement is an application of \eqref{Po-St}, \eqref{El-St}, and a result in \cite{BP} for
two-dimensional domains.  We will ultimately prove (in Section \ref{s3}) that $\Lambda=\emptyset$.

\begin{proposition}\label{cont}
Assume that $\Omega$ is a bounded $C^2$ domain in $\bR^2$.
Assume $u=u(x)$ is in $H^1(\Omega; \overline {\sK})$ and $u$ is a finite energy minimizer of $J$ in $\Omega$.  Let $\Lambda=\{x \in \Omega: u(x) \in \partial \sK\}.$

a) If $\Omega'$ is a connected open set with $\Omega' \subset \subset \Omega$, there exist constants $0<\sigma<1$ and $c_1>0$ such that $\omega(d)=c_1 d^{\sigma}$ is a modulus of continuity for $u$ in $\Omega'$.  The constants $\sigma$ and $c_1$ depend only on $J(u)$, $\Omega$, $\text{dist }(\Omega',\partial \Omega)$, $M_1$, and the constants $M$ and $\lambda$ in \eqref{Po-St} and \eqref{El-St}.

b)  If $u_0 \in H^1(\Omega;\overline \sK)$ and $J(u_0) < \infty$ such that $u_0 \in C^{0,1}(\partial \Omega;\overline \sK)$ and $\int_{\partial \Omega} f(u_0) ds < \infty$ and if $u \in H^1(\Omega;\overline \sK)$ is a minimizer of $J$ in $A_{u_0}=\{v \in H^1(\Omega;\overline \sK):v-u_0 \in H^1_0(\Omega;\mathbb{R}^q)\}$, then there exists a constant $0<\beta<1$ such that $u \in C^{\beta}(\overline \Omega;\overline \sK)$.  The modulus of continuity, $\omega(d)=c_2 d^{\beta}$, has constants depending only on $J(u)$, $\Omega$, $M_1$, $M$, $\lambda$ and $u_0$.

c)  The minimizer $u$ is continuous in $\Omega$ and $C^{2,\delta}$ in the open set $\Omega \setminus \Lambda$ for all $0<\delta<1$.
\end{proposition}

\begin{proof}
In \cite{BP} we investigated finite energy minimizers $\tilde Q(x)$ of a constrained energy of the form
$$\tilde J (Q)=\int\limits_\Omega(F_e(Q,DQ)+f_b(Q))dx$$
over all $Q \in H^1(\Omega;\overline \sM)$ such that $Q-Q_0 \in H^1_0(\Omega;S_0)$  and $\tilde J(Q_0) < \infty$, where $\sM$ is an open bounded convex subset of
$S_0=\{Q \in \mathbb{M}^{3\times3}:Q=Q^t \text{ and tr } Q=0\}$ and $\mathbb{M}^{3\times3}$ is the set of $3\times3$ real-valued matrices.
Note that $S_0$ is isometrically isomorphic to $\mathbb{R}^5$.  Here $f_b(Q)=g_b(Q)-\kappa |Q|^2+b_0$ for $Q \in S_0$ (and $\infty$ otherwise)
and it is assumed that $g_b$ is a smooth convex function defined on $\sM$ such that $g_b(Q) \to \infty$ as $Q \to \partial \sM$ with $Q \in \sM$.
Since $f(v)=f_0(v)-\frac{M}{2} |v|^2$,  our assumptions (\ref{Po-St}) and (\ref{El-St})  on $f(v)$ and $F(v,Dv)$ correspond to the assumptions (1.2) and (1.4) on $f_b(Q)$ and $F_e(Q,DQ)$ in \cite{BP} that were used to prove the same H\"{o}lder continuity on $\tilde Q=\tilde Q(x)$ that we wish to prove here for $u=u(x)$.  The change from energy densities that depend on the variable $Q \in S_0$ to those that depend on $u \in \mathbb{R}^q$ is a trivial one, and the arguments in the proofs of Theorem 1 and 2 go through to prove a) and b).

To prove c), we first note that  by a), $u$ is continuous in $\Omega$ and hence $u^{-1}(\sK) = \Omega \setminus \Lambda$ is an open set.  To verify that $u$ is $C^{2,\delta}$ on this set, we argue as in \cite{BP}, Corollary 2.  Indeed, assume $B_{4r}(x_0) \subset \subset \Omega \setminus \Lambda$.  Note that $f$ is bounded and $C^2$ on a neighborhood of $u(B_{4r}(x_0))$.  We can then take smooth first variations for $J$ about $u$ supported in $B_{4r}(x_0)$ to conclude that $u$ is a weak solution of (\ref{Eq}) on $B_{4r}(x_0)$.  We can apply the result from \cite{G}, Ch. VI, Proposition 1 asserting that in two space dimensions a continuous weak solution of (\ref{El-St})-(\ref{Eq})
with $f_u(u(x))$ bounded is in $W^{2,p}(B_{3r}(x_0))$ for some $p > 2$ and thus its first derivatives are H\"{o}lder continuous on $B_{2r}(x_0)$.  Now we can apply techniques from linear elliptic theory in \cite{G}, Ch. III.  Taking (\ref{El-St}) into account these lead to $u \in C^{2,\delta}(B_{r}(x_0)).$
\end{proof}

Our next objective is to show that $u\in H^2_{\rm{loc}}(\Omega)$.
To define variations $u+tw_l$ that will provide a proof,
we will need several properties of the convex potential
$$f_0(v)=f(v)+\frac{M}{2}|v|^2$$
in a family of cones $\mathcal{C}$ with vertices in the convex set $\sK$.
For ease of notation, assume from now on without loss of generality that $0$ is in $\sK$.  Since $\sK$ is a bounded convex set in $\bR^q$, it is starlike with respect to $0$.  Let $\mathbb{S}^{q-1}=\partial B_1(0)$ where $B_1(0)$ is the open unit ball in $\bR^q$ centered at $0$.  Let  $g:\mathbb{S}^{q-1} \to \bR^+$ be in $C^{0,1}(\mathbb{S}^{q-1})$ such that the map

$$\nu \in \mathbb{S}^{q-1} \to g(\nu) \nu \in \bR^q$$
is a parametrization of $\partial \sK$. Define $0<m_1<m_2$ by
\begin{equation}\label{m-1}
m_1=inf\{g(\nu):\nu \in \mathbb{S}^{q-1}\} \text { and } m_2=sup\{g(\nu):\nu \in \mathbb{S}^{q-1}\}.
\end{equation}
Define $G(x): \overline{B_1(0)} \to \overline \sK$ by
\begin{equation}\label{Gdefn}
G(x)=
\begin{cases}
g(\frac {x}{|x|}) x  & \text{ if $x \neq 0$}\\
0 & \text{ if $x=0$}.\\
\end{cases}
\end{equation}
Thus $G$ is a bi-Lipschitz continuous map from $\overline{B_1(0)}$ onto $\overline \sK$.  Let $\mathcal{Y}_{\mu}=G(B_{\mu}(0))$ for $0<\mu<1$.  Then $\mathcal{Y}_{\mu}$ is an open convex subset of $\sK$ and $\mathcal{Y}_{\mu} \uparrow \sK$ as $\mu \uparrow 1$.  Fix $r_0>0$ and $0 < \mu_0<1$  so that $\overline{B_{r_0}(0)} \subset \mathcal{Y}_{\mu_0}$.

{\bf Definition.} We define a family of cones $\mathcal{C}$ as follows:  For each $v$ in ${\sK} \setminus \overline{B_{r_0}(0)}$, we define the cone $C_v^-$
to be the closed half-cone with vertex $v$, axis containing the ray from $v$ to $0$, and aperture $\alpha=\alpha(v)$ in $(0,\frac {\pi}{2})$ determined by $\sin \alpha = \frac {r_0}{|v|}.$  (See Figure \ref{fig: fig. 1}.)  The cone $C_v^+$ is the reflection of $C_v^-$ about the point $v$, i.e.
$$C_v^+=\{w=v+\xi:\xi \in \mathbb{R}^q \text{ and } v-\xi \in C_v^-\}.$$
We define $\mathcal{C}$ to be the family of all cones, $C_v^-$ and $C_v^+$, with $v$ in ${\sK} \setminus \overline{B_{r_0}(0)}$.

\begin{figure}[h!]
  \centering

  \includegraphics[width=7in]{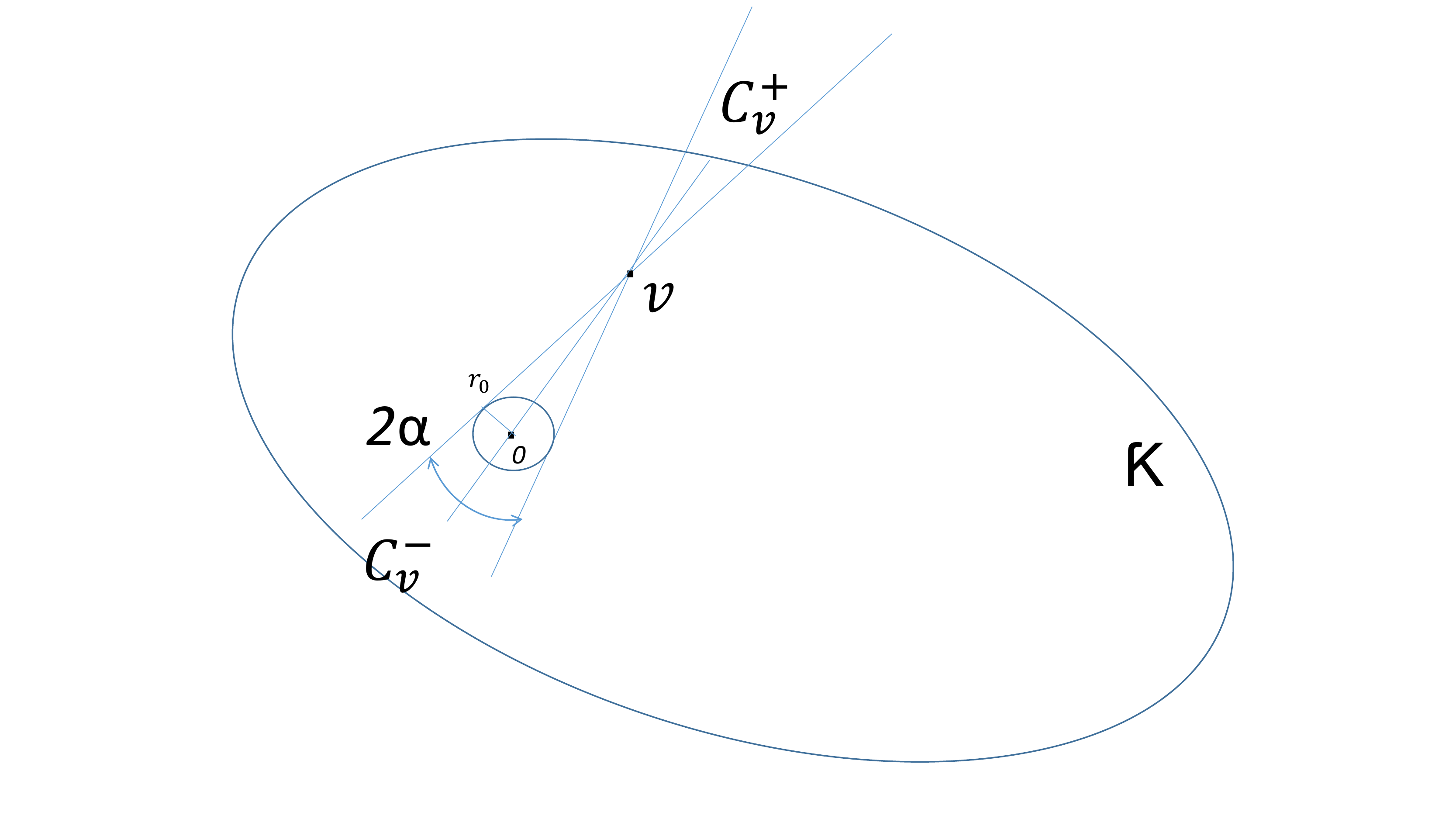}
  \caption{The cones $C_v^-$ and $C_v^+$.}
  \label{fig: fig. 1}
\end{figure}

For $v$ as above, the ball $\overline{B_{r_0}(0)}$ is contained in $C_v^-$ and is tangent to its boundary.  Thus each ray in $C_v^-$ with initial point $v$ intersects $\partial B_{r_0}(0)$ at least once.  Also $r_0<|v| \leq m_2$ and thus there exists $\alpha_0 \in (0,\frac{\pi}{2})$ such that
\begin{equation}\label{alpha0def}
1>\sin \alpha \geq \frac {r_0}{m_2} \equiv \sin \alpha_0 > 0 \text{ for all } v \in  {\sK} \setminus \overline{B_{r_0}(0)}.
\end{equation}

The result below follows from (\ref{alpha0def}) and the symmetry of $C_v^-$ and $C_v^+$.
\begin{proposition} \label{ballprop}
Assume $v \in {\sK} \setminus \overline{B_{r_0}(0)}$, $z$ is a point on the axis of $C_v^+$ with $z \neq v$, and $\gamma$ satisfies $|z-v| \geq \gamma > 0$.  If  $r>0$ satisfies $r \leq \gamma \sin {\alpha_0}$, then
\begin{equation}\label{coneineq}
\overline{B_r(z)} \subset C_v^+.
\end{equation}
\end{proposition}

\begin{proof}
The hypotheses ensure that
$$r \leq \gamma \sin {\alpha} \leq |z-v| \sin {\alpha}$$
for $\alpha=\alpha(v)$.  It follows from this and the symmetry of $C_v^-$ and $C_v^+$ that $\overline{B_r(z)} \subset C_v^+.$
\end{proof}

\begin{proposition}\label{prop23}
If $v \in \sK \setminus \mathcal{Y}_\mu$ and $\mu>\mu_0$, then ${C_v^+ \cap \sK} \subset \sK \setminus {\mathcal{Y}}_\mu$.
\end{proposition}

\begin{proof}
If not, there exists $w \in {C_v^+ \cap \sK}$ such that $w \notin \sK \setminus \mathcal{Y}_\mu$ and hence $w \in \mathcal{Y}_\mu$.  Thus $w \neq v$ and $w=v+\xi$ for some $\xi \neq 0$.  The ray with initial point $v$ that passes through $v-\xi$ is in $C_v^-$ and hence contains a point $y$ in $\partial B_{r_0}(0) \cap C_v^-$.  Since $\mu >\mu_0,
\partial B_{r_0}(0) \subset {\mathcal{Y}}_\mu$. By the convexity of ${\mathcal{Y}}_\mu$, the segment $\overline{yw}$ is contained in ${\mathcal{Y}}_\mu$.  But $v \in \overline{yw} \subset {\mathcal{Y}}_\mu$, which contradicts the fact that $v \in \sK \setminus \mathcal{Y}_\mu$.
\end{proof}

Define $s_0=\text{ max} \{f_0(v): v \in \partial B_{r_0}(0)\}$.  Since $f_0$ is continuous on $\sK$ and $f_0(v) \to \infty$ as $v \to \partial \sK$ with $v$ in $\sK$, there exists a constant $\mu_1$ in $(\mu_0,1)$ such that
\begin{equation}\label{mu1def}
f_0(v) \geq 1 + s_0 \text{ for all } v \text{ in } \sK \setminus \mathcal{Y}_{\mu_1}.
\end{equation}

From now on, we fix $0, r_0, \alpha_0, m_1, m_2, 0 < \mu_0 < \mu_1 <1$, and $s_0$ as above.  We then have the following:

\begin{lemma}\label{lem24}
For any $v$ in $\sK \setminus \mathcal{Y}_{\mu_1}$ and $w$ in $C_v^+ \cap \sK$ such that $w \neq v$, we have $\nabla f_0 (w) \cdot (w-v)>0$.
\end{lemma}

\begin{proof}
If this is false, there exists $v$ in $\sK \setminus \mathcal{Y}_{\mu_1}$ and $w \in C_v^+ \cap \sK$ such that $w \neq v$ and $\nabla f_0 (w) \cdot (w-v)\leq 0$.  By Proposition \ref{prop23}, $w \in \sK \setminus \mathcal{Y}_{\mu_1}$.  Consider a linear path given by $p(t)=v+t\frac {w-v}{|w-v|}$ for $\underline t \leq t \leq \overline t$ where $p(\underline t) \in \partial B_{r_0}(0)$ and $\overline t = |v-w|$.
Setting $h(t)=f_0(p(t))$ we have $h'(\overline t) \leq 0$.
By definition of $m_2$ (see \eqref{m-1}), we have
$$0<\overline t-\underline t = |p(\overline t)-p(\underline t)| = |w-p(\underline t)| \leq |w|+|p(\underline t)| \leq 2m_2.$$
Since $h(t)$ is convex, we then have
$$0 \geq h'(\overline t) \geq \frac{h(\overline t)-h(\underline t)}{\overline t - \underline t} \geq  \frac{h(\overline t)-h(\underline t)}{2m_2}.$$
This is impossible since $m_2>0$ and by definition of $s_0,$
$$h(\overline t)-h(\underline t)=f_0(w)-f_0(p(\underline t)) \geq (1+s_0)-s_0 =1.$$
\end{proof}

\begin{corollary}\label{decf0}
If $v \in \sK \setminus {\mathcal Y}_{\mu_1}$ and $\vec{l}=\vec{wv}$ is a ray in $C_v^+ \cap \sK$ with initial point $w$ and final point $v$, then $f_0$ decreases along $\vec{l}$.
\end{corollary}

We can now prove $H^2_{loc}$ estimates on minimizers using appropriate variations $u+t\phi$ and $u+tw$.

\begin{lemma}\label{h2est}
Assume that $\Omega'$ is an open connected set and $\Omega' \subset \subset \Omega$.  If $u \in  H^1_{loc}(\Omega;\sK)$ is a finite energy local minimizer of $J$ in $\Omega$, then $u \in H^2(\Omega')$ and
$$\|u\|_{H^2(\Omega')} \leq C_o,$$
where $C_o$ depends only on $J(u)$, $\Omega$, $dist(\Omega',\partial \Omega)$, $M_2$, $M$ and $\lambda$.
\end{lemma}

\begin{proof}
Let $\eta_0 >0$ satisfy $3 \sqrt{2}\eta_0 < \frac{1}{4} \text{ dist }(\Omega',\partial \Omega)$.  Define
\begin{equation} \label {defomega''}
\Omega''=\{x \in \Omega: \text{ dist }(x,\Omega') < \frac{1}{2} \text{ dist }(\Omega',\partial \Omega\}
\end{equation}
and let $\omega(d)=C d^{\sigma}$
be a modulus of continuity for $u$ on $\Omega''$.
Given $\eta > 0$ with $\eta < \eta_0$, let $\mathcal{D}_0=\{D_m: m \in \mathbb{N}\}$ be a tiling of $\mathbb{R}^2$ by closed squares of side length $\eta$, so  that $\mathbb{R}^2 = \cup _{m=1}^{\infty}D_m$ and distinct squares in this tiling have at most one edge in common.  For each $m \in \mathbb N$ let $E_m$ be the union of $D_m$ and its eight neighbors.  Thus $E_m$ and $D_m$ have the same centers and $E_m$ has side length $3\eta$.  Set
$$\mathcal{D}_1=\{D_m \in \mathcal{D}_0:D_m \cap \Omega' \neq \emptyset\}.$$
Then $\mathcal{D}_1=\{D_{m_l}: 1 \leq l \leq L\}.$  For ease of notation, let $\tilde {D}_l$ and $\tilde {E}_l$ denote $D_{m_l}$ and $E_{m_l}$, respectively,
for $1 \leq l \leq L$.
Note that by our definition of $\eta_0$, $\tilde {E}_l \subset \Omega''$ for $1 \leq l \leq L$.

We first work through the case $q\geq 2.$ Let $\mu \in [\mu_1,1).$  Thus $\mu_1 \leq \mu \leq \frac{\mu+3}{4} < 1.$ (Later we will also require that $(1-\mu)$ is sufficiently small.)  We partition $\mathcal{D}_1=\mathcal{D}_2 \cup \mathcal{D}_3$ as follows:
\begin{align*}
\tilde{D}_l \in \mathcal{D}_2 \text{ if }&u(\tilde {E}_l) \cap (\overline \sK \setminus \mathcal {Y}_{\frac{\mu+3}{4}}) \neq \emptyset,\\
\tilde{D}_l \in \mathcal{D}_3 \text{ if }&u(\tilde {E}_l) \subset \mathcal {Y}_{\frac{\mu+3}{4}}.
\end{align*}

We first assume $\tilde{D}_l \in \mathcal{D}_2$ and
show that $u \in H^2(\tilde{D}_l)$.  Note that in this case, $u(\tilde {E}_l)$ is not a subset of $\partial \sK$ because if so, $\tilde {E}_l \subset \Lambda$ (by (\ref{Energy}) and (\ref{Po-St}) since $J(u) < \infty$) and this would imply that $|\tilde {E}_l|=0$, a contradiction.  From this and the definition of $\mathcal{D}_2$ it follows that there exists $\overline y \in \tilde {E}_l$ such that $u(\overline y) \in \sK \setminus \mathcal {Y}_{\frac{\mu +3}{4}}.$  Thus
$$u(\overline y)=\tilde {\mu} g(\overline \nu) \overline \nu
\text{ for }\overline \nu=\frac{u(\overline y)}{|u(\overline y)|} \in S^{q-1}$$ and some constant $\tilde \mu$ depending on $\overline y$ such that $\frac{\mu +3}{4} < \tilde \mu < 1.$  Set $v_l=\mu g(\overline \nu) \overline \nu$.  Recall that $\mu_0 < \mu_1 \leq \mu < \frac{\mu +3}{4}$; hence $v_l$ is on the segment between $0$ and $u(\overline y)$, and is contained in $\sK$ by convexity.  Also $v_l \in
\sK \setminus {\mathcal Y}_{\mu_1} \subset \sK \setminus \overline{B_{r_0}(0)}$.  Thus the cone $C_{v_l}^+$ is in our family of cones $\mathcal{C}$ and $u(\overline y)$ is on the axis of $C_{v_l}^+$.
Since $m_2 \geq g(\overline \nu) \geq m_1$ and $1-\mu > \tilde \mu - \mu > \frac{\mu+3}{4} -\mu = \frac{3}{4}(1-\mu)$, using the definition of $u(\overline y)$ and $v_l$,  we have
\begin{align} \label{2.5}
m_2 (1-\mu) & \geq  g(\overline \nu) (\tilde {\mu}-\mu)
  =|u(\overline y)-v_l| \\
& > m_1 (\frac{\mu+3}{4}-\mu) = \frac{3m_1}{4} (1-\mu) \equiv \gamma_1 > 0. \nonumber
\end{align}

Next assume further that $\eta$ satisfies
\begin{equation}\label {2.6}
\omega(3 \sqrt {2} \eta) < \frac{3m_1}{4} (1-\mu) \sin \alpha_0 \equiv r_1 = \gamma_1 \sin \alpha_0.
\end{equation}
Since $\overline y \in \tilde{E_l}$ and the diameter of $\tilde{E_l}$ is $3 \sqrt {2} \eta$, $u(\tilde{E_l}) \subset B_{r_1}(u(\overline y)).$  By \eqref{2.5} and \eqref{2.6},
\begin{equation*}
0 < r_1=\gamma_1 \sin \alpha_0   < |u(\overline y)-v_l| \sin \alpha_0.
\end{equation*}

From this and Proposition \ref{ballprop},
we have $\overline{B_{r_1}(u(\overline y))} \subset C_{v_l}^+$.  Thus
\begin{equation}\label{ballineq}
u(\tilde{E_l}) \subset {B_{r_1}(u(\overline y))
} \subset C_{v_l}^+.
\end{equation}

Let $\zeta_l \in C^2_c((\tilde{E_l})^o)$ such that $\zeta_l=1$ on
$\tilde{D_l}$ and
$0\le\zeta_l\le 1$.  Define $\zeta_l$ to be zero in $\Omega \setminus (\tilde{E_l})^o$.  (For other values $1 \leq j \leq L$, we shall assume that the definition of $\zeta_j$ in $\tilde{E_j}$ differs only by a rigid translation that maps
$\tilde{E_j}$ onto $\tilde{E_l}$).
For $x \in \Omega$, and $h\ne 0$ sufficiently small, define
\begin{equation}\label{defnwl}
w_l(x)=\zeta_l^2(x)(v_l-u(x))|\nabla^h u(x)|^2
\end{equation}
for $1 \leq l \leq L$ where $\nabla^h u = (\nabla_1^h u,\nabla_2^h u)$ with $\nabla_k^h u(x) \equiv h^{-1} [u(x+h e_k) - u(x)]$ for $k=1,2$.
By \eqref{defomega''} and Proposition \ref{cont}, we can choose $\ot(h)>0$ depending only on $h, J(u),$ and $dist (\Omega',\partial \Omega)$ so that
\[\tau\equiv t\zeta^2(x)|\nabla^h u(x)|^2\le 1\quad\text{for } x\in\Omega\quad \text{and } 0\le t\le\ot(h).\]
Note that if $w_l(x) \neq 0$, then $x \in \tilde{E_l} \subset \Omega''$ and by \eqref{ballineq}, $u(x) \in  \overline{\sK} \cap {B_{r_1}(u(\overline y))} \subset \overline{\sK}\cap C_{v_l}^+$. Hence
\begin{align}\label{2.7}
u(x) + tw_l(x)  & = u(x) + \tau (v_l - u(x))\\
 & = (1-\tau) u(x) + \tau v_l \in  \overline{\sK}\cap C_{v_l}^+
\nonumber
\end{align}
by convexity, and it is located on the ray from $u(x)$ to $v_l$ in $ \overline{\sK}\cap C_{v_l}^+$.   By
 Corollary \ref{decf0}, we have
\[f_0(u(x)+tw_l(x))\le f_0(u(x))\text{ for all } x\in\Omega  \text{ and } 0\le t\le \ot(h).\]
Thus
\[\int\limits_\Omega f_0(u+tw_l)dx\le\int\limits_\Omega f_0(u) dx\quad\text{for } 0\le t\le\ot(h)\]
and since $u$ is a minimizer and $f(u)=f_0(u)-\frac{M}{2} |u|^2$,
\[0\le J(u+tw_l)-J(u)\le\int\limits_\Omega((F(u+tw_l, D(u+tw_l))-F(u, Du)+\frac M2(|u|^2-|u+tw_l|^2)) dx.\]
Dividing by $t$ and letting $t \to 0$, we obtain
\begin{align}\label{Est-2}
0 & \leq \int\limits_\Omega(F_P(u,Du): Dw_l+ [F_u(u,Du)-Mu]\cdot w_l) dx\\
  & = \int\limits_{\tilde E_l}(F_P(u,Du): Dw_l+ [F_u(u,Du)-Mu]\cdot w_l) dx.
\nonumber
\end{align}

We now proceed as in the proof of \cite{G}, Ch.~II, Thm 1.2.  Let $C_i$ denote constants that are independent of $h$.  Consider  $\tilde{D}_l \in \mathcal{D}_2$ and $\zeta \equiv \zeta_l$ as above.  The first step is to  consider the test function $\phi \equiv \phi_l = \phi_{l,k}$ defined by
\begin{align*}
\phi_{l,k}(x) &=[\zeta^2(x)u(x+h e_k)+\zeta^2(x-h e_k) u(x-h e_k)-(\zeta^2(x)+\zeta^2(x-h e_k))u(x)]h^{-2}\\
&=\nabla^{-h}_k(\zeta^2\nabla_k^h u)\quad \rm{for}\,\, 1\le k\le 2,
\end{align*}
where $\nabla^h_k u(x)=h^{-1}[u(x+he_k)-u(x)]$ for $h\ne 0$ sufficiently small so that $\phi$ is compactly supported in $(\tilde{E}_l)^\circ$ for $k=1,2$.
Recall that $\tilde{E}_l \subset \Omega''$ and $u(\tilde{E}_l) \subset B_{r_1}(u(\overline y))\subset \overline{\sK}$.  Thus by \eqref{ballineq}, $u(\tilde{E}_l) \subset {\sK} \cap C_{v_l}^+$.  Since $f_0(v)$ is convex on $\sK$,  it was proved in \cite{E} that for $h$ sufficiently small and $0 < t < t(h)$ sufficiently small, we have
$$\int_{\tilde{E}_l} (f_0(u+t\phi_{l,k}) - f_0(u)) dx \leq 0.$$  As in \eqref{Est-2}, since $u$ is a minimizer of $J$ we have
\begin{align}\label{Est-3}
0 & \leq \int\limits_\Omega(F_P(u,Du): D\phi_{l,k}+ [F_u(u,Du)-Mu]\cdot \phi_{l,k}) dx\\
  & = \int\limits_{\tilde E_l}(F_P(u,Du): D\phi_{l,k}+ [F_u(u,Du)-Mu]\cdot \phi_{l,k}) dx. \nonumber
\end{align}
Here
\begin{align*}
&\int\limits_{\tilde E_l}\ (F_P(u,Du): D\phi_{l,k})dx
 = \int\limits_{\tilde E_l}(F_{p^i_{\alpha}}(u,Du)\cdot {\frac {\partial}{\partial x_{\alpha}}(\nabla^{-h}_k}[\zeta^2 \nabla^h_k u^i]) dx\\
  & = \int\limits_{\tilde E_l}[A_{ij}^{\alpha \beta}(u) \frac{\partial u^j} {\partial x_{\beta}} + B^{\alpha}_i(u)]\cdot {\frac {\partial}{\partial x_{\alpha}}(\nabla^{-h}_k}[\zeta^2 \nabla^h_k u^i]) dx \\
  & = -\int\limits_{\tilde E_l}\{\nabla^h_k [A_{ij}^{\alpha \beta}(u) \frac{\partial u^j} {\partial x_{\beta}} + B^{\alpha}_i (u)]\} \cdot  \frac {\partial}{\partial x_{\alpha}}(\zeta^2 \nabla^h_k u^i) dx \\
   & -\int\limits_{\tilde E_l} \{\nabla^h_k [A_{ij}^{\alpha \beta}(u)\} \cdot
   \frac{\partial u^j} {\partial x_{\beta}} + B^{\alpha}_i (u)] \cdot \{ \zeta^2 \nabla^h_k (\frac {\partial u_i}{\partial x_{\alpha}}) + 2\zeta (\frac {\partial \zeta}{\partial x_{\alpha}}) (\nabla^h_k u_i)\} dx
 \end{align*}

and

\begin{align*}
&\int\limits_{\tilde E_l} ([F_u(u,Du) -Mu]  \cdot \phi_{l,k})dx \\
&= \int\limits_{\tilde E_l} [D_u(A_{ij}^{\alpha \beta}(u)) \cdot
   \frac{\partial u^i} {\partial x_{\alpha}} \cdot \frac{\partial u^j}
    {\partial x_{\beta}}  + D_u( B^{\alpha}_i (u)) \cdot \frac{\partial u^i} {\partial x_{\alpha}} -Mu] \cdot [\nabla^{-h}_k(\zeta^2 \nabla^h_k u)] dx  \\
& = - \int\limits_{\tilde E_l} \nabla^h_k\{ [D_u(A_{ij}^{\alpha \beta}(u)) \cdot
   \frac{\partial u^i} {\partial x_{\alpha}} \cdot \frac{\partial u^j}
    {\partial x_{\beta}}  + D_u( B^{\alpha}_i (u)) \cdot \frac{\partial u^i} {\partial x_{\alpha}} -Mu]\}\cdot (\zeta^2 \nabla^h_k u)) dx
    \end{align*}
for $1\leq k \leq 2$.  It follows that

\begin{equation} \label{ineq-1}
\frac{\lambda}{2}\int\limits_{\tilde E_l} |\nabla^h Du|^2\zeta^2\ dx \le C_0\int\limits_{\tilde E_l}|\nabla^h u|^2 |Du|^2\zeta^2\ dx + C_1
\end{equation}
for all $\eta$ sufficiently small,
where $\lambda$ is the constant defined in our assumption (\ref{El-St}).
Next we use \eqref{Est-2} and our definition of $w_l$ to prove a second inequality that will provide an upper bound on the second derivatives of $u$ in $L^2(\tilde E_l)$.

By \eqref{Est-2} and our definition of $w_l$, we have
\begin{align} \label{wl-ineq2}
0 & \leq \int\limits_{\tilde E_l}(F_P(u,Du): Dw_l+ [F_u(u,Du)-Mu]\cdot w_l) dx\\ \nonumber
 & = \int\limits_{\tilde E_l}[A_{ij}^{\alpha \beta}(u)\cdot \frac{\partial u^j}
    {\partial x_{\beta}} + B^{\alpha}_i (u)] \cdot \frac{\partial}{\partial x_{\alpha}}(\zeta^2 (v_l^i-u^i(x))|\nabla^hu|^2) dx \\ \nonumber
 & + \int\limits_{\tilde E_l}[D_u(A_{ij}^{\alpha \beta}(u)) \cdot \frac{\partial u^i} {\partial x_{\alpha}} \cdot \frac{\partial u^j}
    {\partial x_{\beta}} + (D_u(B^{\alpha}_i (u)) \cdot \frac{\partial u^i} {\partial x_{\alpha}} - Mu] \cdot \zeta^2 (v_l-u(x)) |\nabla^h u|^2 dx. \nonumber
\end{align}

Recall that $u(\tilde E_l) \subset B_{r_1}(u(\overline y))$ and by \eqref{2.5} and \eqref{2.6},

\begin{align*}
|v_l-u(x)| & \leq |v_l-u(\overline y)| + |u(\overline y)-u(x)| \leq |v_l-u(\overline y)| + r_1\\
&  \leq m_2(1-\mu) + \frac{3m_1}{4} \sin \alpha_0 (1-\mu)  \equiv C_2 (1-\mu).
\end{align*}

for all $x$ in $\tilde E_l$.
Using this we see from (\ref{wl-ineq2}) that
\begin{equation}\label{ineq-2}
\frac{\lambda}{2}\int\limits_{\tilde E_l}|\nabla^h u|^2 |Du|^2\zeta^2\ dx\le C_3 (1-\mu) \int\limits_{\tilde E_l} |\nabla^h Du|^2\zeta^2\ dx  + C_4
\end{equation}

Note that $C_0$ and $C_3$ depend only on $\lambda$ and $M_2.$  Taking $(1-\mu)$ sufficiently small so that $C_3 (1-\mu) \leq \frac {\lambda^2}{8C_0}$, it follows from \eqref{ineq-1} and \eqref{ineq-2} that

\begin{equation}\label{ineq-3}
\int\limits_{\tilde D_l} (|\nabla^h Du|^2+|\nabla^h u|^2 |Du|^2) dx \le C_5.
\end{equation}

The constants $C_j$ are uniform in $h$ for $|h|\leq h_0$. Letting $h\to 0$ we get

\begin{equation}\label{ineq-4}
\int\limits_{\tilde D_l} (|D^2 u|^2+|Du|^4) dx \le C_5.
\end{equation}

This inequality holds for all $\tilde D_l \in {\mathcal{D}}_2$.
If $\tilde D_l \in {\mathcal{D}}_3$ then $f_u$ is bounded on a neighborhood of the union of all such squares.   It follows as in the proof from \cite{G}, Ch.~II referred to above that we have \eqref{ineq-4} (with a possibly larger value of $C_5$) in this case as well.

In conclusion, we first fix $\mu \in (\mu_1,1)$ sufficiently close to 1 so that \eqref{ineq-3} follows from \eqref{ineq-1} and \eqref{ineq-2}.  We then choose $\eta \in (0,\eta_0)$ sufficiently small so that \eqref{2.6}, \eqref{ineq-1}, and \eqref{ineq-2} hold.
This fixes the covering ${\mathcal{D}}_1$ such that \eqref{ineq-4} holds for a fixed constant $C_5$ for all $1 \leq l \leq L$.  Summing on $l$ we have $\int_{\Omega'} |D^2u|dx <  \infty$.

We now comment on the case $q=1.$  In this instance $\mathcal{K}$ is a bounded interval $(a,b)$ such that $f_0(v)$ is a convex function satisfying $\underset{v\uparrow b}\lim f_0(v)=\infty=\underset{v\downarrow a}\lim f_0(v).$ It follows that there exists $\delta>0$ so that $f_0'(v)>0$ for $b-\delta<v<b.$ In the same way we have that $f_0'(v)<0$ for $a<v<a+\delta.$  With this information we can carry out the argument as above using half lines $R_v^- (R_v^+)$ in place of the  cones $C_v^- (C_v^+)$.
 \end{proof}

\section{Proof of Theorem \ref{uthm}.}\label{s3}

Recall that $u \in C^2(\Omega_0)$ and $\Omega_0 \subset \Omega \setminus \Lambda$.  The two facts that $u\in H^2_{\text{loc}}(\Omega)$ and $u$ satisfies  equation (\ref{Eq}) on $\Omega_0$ imply that $u$ is a strong solution to the equilibrium equations (\ref{Eq}) throughout $\Omega$
and that each term appearing in this equation  is in $L^2_{\text{loc}}(\Omega)$. We use this to prove Theorem \ref{uthm}.

\begin{proof}
Suppose that $B_{2R}(x_0)\subset\subset \Omega$. Let $(\rho,\theta)$ be polar coordinates centered at $x_0$ and consider the field of pure second derivatives $D^2_{\upsilon\upsilon}u(x)$ where $\upsilon=\upsilon(x)=e_{\theta}$ for $x\in B_R(x_0)$. We have $D^2_{\upsilon\upsilon}u(x)=\frac{u_\rho}{\rho}+\frac{u_{\theta\theta}}{\rho^2}$ is in $L^2(B_R(x_0))$. Let $\zeta=\zeta(\rho)\in C^2_c(B_R(x_0))$ such that $\zeta=1$ on $B_{R/2}(x_0)$.  Fix $0<r\le\frac R2$.  Multiplying the equation (\ref{Eq}) by $\zeta^2 D^2_{\upsilon\upsilon} u$ and integrating  over $B_R(x_0)\setminus B_r(x_0)$, we obtain

\begin{equation}\label{Eq2}
\int\limits_{B_R(x_0)\setminus B_r(x_0)} (\text{div} F_P-F_u)\cdot\zeta^2 D^2_{\upsilon\upsilon} u\, dx\,=\,\int\limits_{B_R(x_0)\setminus B_r(x_0)}f_u(u)\cdot\zeta^2(\frac{u_\rho}{\rho}+\frac{u_{\theta\theta}}{\rho^2})\, dx
\end{equation}

Since $\text{div} F_P -F_u$ and the test function $D^2_{\upsilon\upsilon} u$ are in $L^2(B_R(x_0))$ we get
\begin{equation}\label{In1}
\left|\int\limits_{B_R(x_0)\setminus B_r(x_0)} (\text{div} F_P-F_u)\cdot\zeta^2 D^2_{\upsilon\upsilon} u\, dx\right| \le C_1
\end{equation}
where $C_1$  is independent of $r$ for $r$ in $(0,\frac {R}{2}]$.
By (\ref{Eq}), $f_u$ is also in $L^2(B_R(x_0))$.  Consider
\[\int\limits_{B_R(x_0)\setminus B_r(x_0)} f_u(u)\cdot\zeta^2 \frac {u_{\theta\theta}}{\rho^2}\, dx =\int_r^R\int_0^{2\pi}f_u(u)\cdot\zeta^2 \frac{u_{\theta\theta}}{\rho}\,d\theta d\rho.\]
Since $u$ is a finite energy minimizer of $J$ in $\Omega$, $f(u)$ is in $L^1(\Omega)$ and for almost every $\rho$ satisfying $r\le \rho\le R$ we have
\begin{equation}\label{In2}
\int^{2\pi}_0 (|f(u(\rho,\theta))|+|f_u(u(\rho,\theta))|^2+|u_\theta(\rho,\theta)|^2) d\theta<\infty.
\end{equation}
For such a value of $\rho$, if there is an interval $(\alpha,\beta)\subset [0, 2\pi)$ so that $u(\rho,\beta)\in \Lambda$ and $u(\rho,\theta)\notin\Lambda$ for all $\theta$ in $[\alpha,\beta)$, then
\[\int^\theta_\alpha \partial_\phi f(u(\rho,\phi)) d\phi=\int^\theta_\alpha f_u(u(\rho,\phi)\cdot u_\phi d\phi \le C_1 < \infty.\]
Thus $f(u(\rho,\theta))-f(u(\rho,\alpha))\le C_1<\infty$. However  $\lim\limits_{\theta\to\beta} f(u(\rho,\theta))=\infty$ and this is not possible. It follows that for each $\rho$ for which (\ref{In2}) holds,  we have $\partial B_\rho(x_0)\cap\Lambda=\emptyset$. Since $u\in C^2(\Omega\setminus \Lambda)$ it follows that the functions $u(\rho,\cdot)$ and $f(u(\rho,\cdot))$ are smooth.  This allows us to integrate by parts and obtain
\begin{align*}
\frac{\zeta^2(\rho)}{\rho} \int^{2\pi}_0 f_u\cdot u_{\theta\theta} d\theta&= -\frac{\zeta^2(\rho)}{\rho} \int^{2\pi}_0 u_\theta \cdot D^2f\cdot u_\theta\, d\theta\\
&\le M\int\limits_{\partial B_\rho(x_0)} |Du|^2 ds.
\end{align*}
We conclude then that
\begin{equation}\label{In3}
\int\limits_{B_R(x_0)\setminus B_r(x_0)} f_u\cdot{\frac{ u_{\theta\theta}}{\rho^2}} \zeta^2 dx \le M \int_{B_R(x_0)} |Du|^2 dx\le C_2.
\end{equation}
Next using the estimate
\[\int_{B_R(x_0)}(|f_u(u(x))|^2+|Du(x)|^2) dx <\infty\]
and the same argument as above on almost every line parallel to one of the coordinate axes it follows that $f(u(x))\in W^{1,1}(B_R(x_0))$ and $Df(u)=f_u\cdot Du$ almost everywhere  on $B_R(x_0)$. With this fact we see that
\begin{align*}
&\int_{B_R(x_0)\setminus B_r(x_0)}\zeta^2 f_u(u(x))\cdot\frac{u_{\rho}}{\rho} dx\\
&=\int^{2\pi}_0\int^R_r\zeta^2(\rho)(f(u(\rho,\theta)))_\rho d\rho d\theta\\
&=-r^{-1}\int_{\partial B_r(x_0)}f(u)ds
-\int_{B_R(x_0)\setminus B_{R/2}(x_0)} f(u)\frac{(\zeta^2)_\rho}{\rho} dx
\end{align*}
where for the last term we used the fact that $\zeta^2=1$ on $B_{R/2} (x_0)$.

\noindent
Thus
\begin{equation}\label{In4}
\int_{B_R(x_0)\setminus B_r(x_0)}\zeta^2 f_u\cdot\frac{u_\rho}{\rho} dx\le -r^{-1}\int_{\partial B_r(x_0)} f ds+ C_3.
\end{equation}
Taking \eqref{In1}, \eqref{In3}, and \eqref{In4} together with \eqref{Eq2} we see that

\[r^{-1}\int_{\partial B_r(x_0)} f(u)ds \le C_4\quad \text{for } 0<r<\frac R2\]
where $C_4$ is independent of $r$.
Since $u(x)$ is continuous on $\Omega$ and $\lim_{u \to u_0} f(u)=f(u_0) \in (-\infty,\infty]$ for each $u_0 \in \overline{\sK}$,  we have
$$2\pi f(u(x_0)) \leq C_4,$$
where $C_4$ depends on $R$, $J(u)$, and $\|u\|_{H^2(B_R(x_0)}$.  In particular $\Lambda=\emptyset$.  Furthermore, applying  Lemma \ref{h2est} we see that $dist(u(x_0), \partial \sK) \geq c > 0$
where $c$ depends on $dist(x_0,\partial \Omega)$ and $J(u)$.
\end{proof}

\section{Applications to Liquid crystals.}\label{s4}
We briefly describe the liquid crystal model that motivates our constrained problem and state our result as it applies to this case. A more detailed overview of energies for liquid crystals is given in \cite{B}. Let $\Omega\subset\bR^3$ be a region filled with rod--like liquid crystal molecules. For $x\in\Omega$ and $p\in \mathbb{S}^2$ denote by $\rho(x,p)$ the probability distribution for the long axes of the molecules near $x$ aligned with the direction $p$. We have
\begin{equation}\label{rho}
\rho(x,p)\ge 0,\quad \int\limits_{\mathbb{S}^2}\rho(x,p)dp=1.
\end{equation}
If the directions are random so that no direction is preferred then $\rho(x,p)=\rho_0=\frac 14\pi$ and the liquid crystal is in the isotropic state  at $x$. The de Gennes $Q$ tensor is introduced as a macroscopic order parameter
\begin{equation}\label{Q}
Q(x)=\int\limits_{\mathbb{S}^2}(p\otimes p-\frac 13 I)\rho(x,p) dp
\end{equation}
representing the second moments of $\rho$ normalized so that $Q=0$ if $\rho=\rho_0$.
Set $S_0=\{A\in \bM^{3\times 3}: A=A^t, tr A=0\}$. Then from (\ref{rho}) and (\ref{Q}) we see that $Q$ takes on values in the open, bounded and convex set
\begin{equation}
\sM=\{A\in S_0: -\frac 13 <\lambda_{\min}(A)\le \lambda_{\max}(A)<\frac 23\},
\end{equation}\label{43}
where $\lambda_{\min}(A)$ and $\lambda_{\max}(A)$ are the minimum and maximum eigenvalues of $A$ respectively. This is the set of physically attainable states in $S_0$. A free energy for constant nematic liquid crystal states identified with $Q\in\sM$ was developed by Katriel, Kventsel, Luckhurst and Sluckin \cite{K} and Ball and Majumdar \cite{BM}. Assuming Maier--Saupe molecular interactions it takes the form
\begin{equation}\label{44}
\begin{cases}
\psi_b(Q)=Tf_{ms}(Q)-\kappa|Q|^2\quad &\text{for } Q\in\sM,\\
\quad\quad{\ }=\infty\quad &\text{for } Q\in S_0\setminus \sM,\\
f_b(Q)={{\text{inf}}_{\rho\in A_Q}} (\int\limits_{\mathbb{S}^2}\rho\, \log \rho\, dp)
\end{cases}
\end{equation}
where
\[A_Q=\{\rho\in L^1(\mathbb{S}^2):\, \rho\ge 0, \int\limits_{\mathbb{S}^2}\rho(p)dp=1,\quad Q=\int\limits_{\mathbb{S}^2}(p\otimes p-\frac 13 I)\ \rho(p)\ dp\}\]
and $T, \kappa>0$. It is shown in \cite{BM} that $f_{ms}$ is convex on $\sM$ with $\lim\limits_{Q\to\partial \sM} f(Q)=\infty$ and it  is shown in \cite{F} that $f\in C^\infty(\sM)$. Ball and Majumdar then used $\psi_b$ to define an energy functional to characterize stable spatially varying liquid crystal configurations. They considered local minimizers $Q\in H^1(\Omega,\overline{\sM})$ to
\begin{equation}
I_{LdG}[Q]=\int\limits_\Omega(G(Q, DQ) +\psi_b(Q)) dx
\end{equation}
where following \cite{L}  the elastic energy density takes the form
\begin{equation}\label{45}
G(Q, DQ)=\sum^5_{i=1} L_i I_i(Q, DQ),
\end{equation}such that
\begin{eqnarray*}
 I_1=D_{x_k} Q_{ij} D_{x_k}Q_{ij}&\quad\quad& I_2=D_{x_j} Q_{ij} D_{x_k}Q_{ik},\\
I_3=D_{x_j} Q_{ik} D_{x_k}Q_{ij}&\quad\quad& I_4=Q_{\ell k} D_{x_\ell} Q_{ij} D_{x_k}Q_{ij},\\
I_5=\epsilon_{\ell j k}Q_{\ell i} D_{x_j} Q_{ki}.&&
\end{eqnarray*}
\noindent
These are polynomial expressions in terms of $Q$ and $DQ$ that satisfy the principle of frame indifference and material symmetry. The expressions $I_1,I_2,I_3 \text{ and }I_4$ satisfy both properties. Analytically this means that $I_1, \cdots,I_4$ are invariant under transformations over $O(3)$ while  $I_5$ satisfies only frame indifference and is invariant under transformations over $SO(3)$. The first four terms are quadratic in $DQ$ while the fifth is linear and is included in the elastic energy density when modeling chiral liquid crystals.  Here ${\epsilon_{\ell j k}}$ is the Levi--Civita tensor.

Let
$\sD:=\{D=[D_{ijk}]\quad 1\le i,j,k\le 3 : D_{ijk}=D_{jik}$ and $\sum\limits_{i=1}^3 D_{\ell\ell k}=0$ for each $i, j$ and $k\}$.
The elasticity constants are to be chosen so that $\sum\limits_{\ell=1}^4 L_k I_\ell(Q, D)\ge c_0|D|^2$ for some $c_0>0$ for all $Q\in \overline{\sM}$ and $D\in\sD$. Inequalities that ensure the coercivity condition are
\begin{equation}\label{47}
L'_1+\frac 53 L_2+\frac 16 L_3>0,\quad L'_1-\frac 12 L_3>0, \quad L'_1+L_3>0
\end{equation}
where
\[L'_1=\begin{cases}
L_1-\frac 13 L_4 &\quad\text{if $L_4\ge 0$}\\
L_1+\frac 23 L_4 &\quad\text{if $L_4\le 0$}.\end{cases}\]
(See \cite{B}).

The space $S_0$ has dimension five. If we take an orthonormal basis $\{E_1, \dots, E_5\}$ we can parameterize the space with the isometry
\[Q(v):\bR^5\to S_0,\quad Q(v)=\sum\limits^5_{j=1} v_jE_j.\]
Set $\sK=\{v\in\bR^5 : Q(v)\in \sM\}$, $f(v)=\psi_b(Q(v))$ and $F(v, Dv)=G(Q(v), DQ(v))$. Then $\sK$ is an open, bounded and convex region in $\bR^5$, $f$ satisfies  (\ref{Po-St}) and $F$  satisfies   (\ref{El-St}).
For the case $n=2$ we view the liquid crystal body on the infinite cylinder $\Omega\times\bR$ where $\Omega\subset \bR^2$ is the cylinder's cross--section and the order parameter is $Q=Q(x_1, x_2)$. We can then apply our results from Theorem \ref{uthm} to $J[v]=I_{LdG}[Q(v)]$ to obtain:

\begin{theorem}\label{thm41}
Let $\Omega\subset \bR^2$ and let $Q\in H^1(\Omega; \overline{\sM})$ be a finite energy local minimizer for $I_{LdG}[\cdot]$ satisfying (\ref {44})-(\ref{47}). Then $Q\in C^2(\Omega)$. If $E\subset\subset \Omega$ then  $Q(E)\subset\subset \sM$ and $Q$ satisfies the equilibrium equation
\[[ { \textmd{div} }\,G_D(G, DQ) - G_Q(Q, DQ)-\psi_{b,Q}(Q)]^{st}=0 \text{ in }\Omega.
\]
where $[A]^{st}$ is the symmetric and traceless part of $A\in \bM^{3\times 3}$.
\end{theorem}
\bigskip
\noindent{\bf Acknowledgement}. {This work was  partially supported by NSF grant DMS-1412840.}


\begin{thebibliography}{99}

\bibitem{B1}
J.M.~Ball. \emph{ Analysis of liquid crystals and their defects}. Lecture notes of the Scuola Estiva GNFM,
Ravello 17–22 Sep. 2018.

\bibitem{B}
J.M.~Ball. \emph{Mathematics and liquid crystals}. Molecular Crystals and Liquid Crystals,\,{\bf 647}(1): 1--27, 2017.

\bibitem{BM}
J.M.~Ball and A.~Majumdar. \emph{Nematic liquid crystals: from Maier--Saupe to continuum theory}. Molecular Crystals and Liquid Crystals,\,{\bf 525}: 1--11, 2010.

\bibitem{BP}
P.~Bauman and D.~Phillips. \emph{Regularity and behavior of eigenvalues for minimizers of a constrained $Q$--tensor energy for liquid crystals}. Cals. Var. Partial Differential Equations,\,{\bf 55}(4), Art. 81, 22 pp, 2016.

\bibitem{E}
L.C.~Evans, O.~Kneuss and H.~Tran. \emph{Partial regularity for minimizers of singular energy functionals, with applications to nematic liquid crystal models}. Trans. Amer. Math. Soc.,\,{\bf 368} (5): 3389--3413, 2016.

\bibitem{F}
E.~Feireisl, E.~Rocca, G.~Schimperna and A.~Zarnescu. \emph{Nonisothermal nematic liquid crystal flows with the Ball--Majumdar free energy}. Ann.~Mat.~Puva.~Anal.,\, (4) {\bf 194}, no.~50,  1269--1299, 2015.

\bibitem{G}
M.~Giaquinta. \emph{Multiple Integrals in the Calculus of Variations and Nonlinear Elliptic Systems}. Annals of Mathematics Studies,\, {\bf 105}, Princeton University Press. Princeton, NJ, 1983.

\bibitem{GT}
Z.~Geng and J.~Tong. \emph{{Regularity of a tensor-valued variational obstacle problem in three dimensions}. Calc. Var. Partial Differential Equations},\,{\bf 59}, Art. 57,  2020.

\bibitem{K}
J.~Katriel, G.F.~Kventsel, G.~Luckhurst and T.~Sluckin. \emph{Free energies in the Landau and molecular field approaches}. Liquid Crystals,\, {\bf 1} 337--355, 1986.

\bibitem{L}
L. Longa, D.~Monselesan, and H.~Trebin. \emph{An extension of Landau de Gennes theory for liquid crystals}. Liquid Crystals,\, {\bf 2} no.~6, 769--796, 1987.


\end{thebibliography}
\end{document}